\documentclass[11pt]{amsart}  
\usepackage{amssymb,amsmath,amsthm,amscd}

\usepackage[hypertex]{hyperref}

\numberwithin{equation}{section}

\newtheoremstyle{fancy1}{10pt}{10pt}{\itshape}{12pt}{\textsc\bgroup}{.\egroup}{8pt}{
}
\newtheoremstyle{fancy2}{10pt}{10pt}{}{12pt}{\itshape}{.}{8pt}{ }

\theoremstyle{fancy1}

\newtheorem{prop}[equation]{Proposition}
\newtheorem{thm}[equation]{Theorem}
\newtheorem*{thm*}{Theorem}

\newtheorem{main}{Theorem}
\newtheorem*{main*}{Theorem}

\newtheorem*{cor*}{Corollary}
\newtheorem*{prop*}{Proposition}
\newtheorem*{remark*}{Remark}

\newtheorem*{problem*}{Problem}

\theoremstyle{fancy2}

\newtheorem*{rems*}{Remarks}

\newtheorem*{rem*}{Remark}

\newtheorem*{example*}{Example}

\newcommand{\cref}[1]{Corollary~\ref{#1}}

\newcommand{\pref}[1]{Proposition~\ref{#1}}

\newcommand{\Sph}{\mathbb{S}}

\newcommand{\C}{{\mathbb{C}}}
\newcommand{\R}{{\mathbb{R}}}
\newcommand{\Z}{{\mathbb{Z}}}

\newcommand{\QH}{{\mathbb{H}}}

\renewcommand{\H}{H}

\newcommand{\G}{\ensuremath{\operatorname{G}}}

\newcommand{\SO}{\ensuremath{\operatorname{SO}}}

\renewcommand{\O}{\ensuremath{\operatorname{O}}}
\newcommand{\Sp}{\ensuremath{\operatorname{Sp}}}
\newcommand{\U}{\ensuremath{\operatorname{U}}}
\newcommand{\SU}{\ensuremath{\operatorname{SU}}}
\newcommand{\Spin}{\ensuremath{\operatorname{Spin}}}

\renewcommand{\S}{\ensuremath{\operatorname{S}}}

\newcommand{\fg}{{\mathfrak{g}}}

\newcommand{\fh}{{\mathfrak{h}}}
\newcommand{\fm}{{\mathfrak{m}}}

\newcommand{\fp}{{\mathfrak{p}}}

\newcommand{\ft}{{\mathfrak{t}}}

\def\con#1=#2(#3){#1 \equiv #2 \bmod{#3}}

\newcommand{\ml}{\langle}                     
\newcommand{\mr}{\rangle}                    

\newcommand{\diag}{\ensuremath{\operatorname{diag}}}

\newcommand{\rank}{\ensuremath{\operatorname{rk}}}

\newcommand{\Ad}{\ensuremath{\operatorname{Ad}}}

\DeclareMathOperator{\Fix}{Fix}

\DeclareMathOperator{\Id}{Id}

\newcommand{\Hi}{H^\iota}
\newcommand{\His}{H^{\iota_2}}

\DeclareMathOperator{\Fi}{\mathrm{Fix}(\iota)}
\DeclareMathOperator{\Fis}{\mathrm{Fix}(\iota_2)}

\begin{document}

\title{Reversible homogeneous Finsler metrics with positive  flag curvature}
\author{Ming Xu}
\address{Tianjin Normal University\\
       Tianjin 300387\\
       P.R.China}
\email{mgmgmgxu@163.com}

\author{Wolfgang Ziller}
\address{University of Pennsylvania\\
       Philadelphia, PA 19104\\
       U.S.A}
\email{wziller@math.upenn.edu}
\thanks{The first author was supported by NSFC (no.\!\! 11271216), Science and Technology Development Fund for Universities and Colleges in Tianjin (no. 20141005), and Doctor fund of Tianjin Normal University (no. 52XB1305).
 The second author was supported by a grant from the National Science Foundation.}

\maketitle

Many of the known results about  Riemannian manifolds with positive sectional curvature  carry over to the setting of Finsler manifolds with positive flag curvature, essentially using the same proof, see \cite{BCS}. Exceptions are those results that require explicit formulas for the curvature. Part of the difficulty is that sectional curvature needs to be replaced by the flag curvature, which depends not only on the plane, but also on a given vector in the plane. Thus the classification of Riemannian homogeneous manifolds with positive sectional curvature \cite{Be,Wa,AW,BB,WZ} does not carry over immediately to the Finsler setting. A classification for even dimensional positively curved homogeneous Finsler manifolds was carried out in \cite{Finslereven}, and for positively curved normal homogeneous Finsler manifolds in \cite{XD2014}. Assuming the metric is reversible,  we have a partial classification for the odd dimensional ones in \cite{Finslerodd} as well. Left over are those odd dimensional homogeneous spaces $G/H$ where the Lie algebra $\fh$ is a regular subalgebra of $\fg$. In this paper we study this remaining case.

\begin{main} Let $M=G/H$ be a compact simply connected homogeneous space with a reversible Finsler metric with positive flag curvature on which the compact Lie group $G$ acts by isometries. If $M$ is odd dimensional and $\fh$ is a regular subalgebra of $\fg$, then either $G/H$ also carries a Riemannian homogenous metric with positive sectional curvature, or $G/H$ is one of
\begin{itemize}
\item[(a)] $\Sp(2)/\diag(z,z^3)$ with $z\in\C$;
\item[(b)]  $\Sp(2)/\diag(z,z)$ with $z\in\C$;
\item[(c)] $\Sp(3)/\diag(z,z,r)$ with $z\in\C,\ q\in\QH$;
\item[(d)] $\SU(4)/\diag(zA,z,\bar z^3)$ with $A\in\SU(2)$ and $z\in\C$;
\item[(e)] $G_2/\SU(2)$ with $\SU(2)$ the normal subgroup of $\SO(4)$ corresponding to the long root.
\end{itemize}
\end{main}
Recall that a Finsler manifold $(M,F)$ is called reversible if $F_x(u)=F_x(-u)$, for any $x\in M$ and $u\in T_xM$. Also recall call $\fh$ a is regular subgroup of $\fg$, if, with respect
to appropriately chosen Cartan subalgebras, all roots of $\fh$ are restrictions of roots of $\fg$.

We do not know wether the remaining examples in (a)-(d)  carry a homogeneous Finsler metric with positive flag curvature since our methods do not easily apply, see Section 3.

Combing with  \cite{Finslereven,Finslerodd}  we obtain:

\begin{cor*}
Let $M=G/H$ be a compact simply connected homogeneous  space with a reversible $G$-invariant Finsler metric with positive flag curvature. Then $G/H$ also carries a Riemannian homogenous metric with positive sectional curvature, or belongs to one of the spaces in Theorem A.
\end{cor*}

For a complete list of  Riemannian homogenous spaces with positive sectional curvature see e.g. \cite{WZ,Zi}. There one also finds further geometric information about these spaces, including a classification in the case the homogeneous space is not simply connected.

If one restricts the class of Finsler metrics to one that often arises in physics and other applications, we   are able to exclude three of the exceptions in Theorem A. For this recall that an $(\alpha,\beta)$ Finsler metric is determined by a Riemannian metric $\ml \cdot,\cdot\mr$ with norm $|\cdot |$, a vector field $v_0$  and a positive function $\phi$ such that $F(v)=|v| \phi(\ml v_0,v/|v|\mr)$. In order for $F$ to be reversible, we
choose $\phi$ to be an even function.

\begin{main} Let $M=G/H$ be a compact simply connected homogeneous space which admits a reversible $G$-invariant
 $(\alpha,\beta)$ Finsler metric with positive flag curvature. Then  $G/H$ also carries a Riemannian homogenous metric with positive sectional curvature, or  is one of $\Sp(2)/\diag(z,z)$ or $\Sp(3)/\diag(z,z,r)$ with $z\in\C,\ q\in\QH$.
\end{main}

\bigskip

The methods we use are a combination of the ones in \cite{WZ} and \cite{Finslereven,Finslerodd}. The main technique is an induction on the dimension of the Lie group by looking at fixed point sets of involutive isometries, as in \cite{WZ}.   But the obstructions in the Riemannian setting, in particular  the existence of commuting eigenvectors or the Block Lemma in \cite{WZ}, does not carry over easily to the Finsler setting. Without these obstructions,  the induction proof reduces the possibilities to a short list of low dimensional homogeneous spaces, 2 examples and 2 infinite families, see Section 3. Notice that this reduction procedure does not use the assumption that the Finsler metric is reversible. For these remaining examples, one applies formulas for the flag curvature for some special flags obtained in \cite{Finslereven}.

There has been  recent progress on the  flag curvature   for  homogeneous Finsler manifolds \cite{XD3}, where one finds an obstruction to positive flag curvature of a homogeneous space without the
assumption of reversability, which hopefully will enable one to carry out a classification in the general case as well.

\section{Preliminaries}

We assume that $G/H$ is a compact simply connected  homogeneous space, with compact connected $G$ and $H$, and that the
 normal subgroup common to both is at most finite. We can also assume,  by making the action ineffective if necessary, that the semisimple part of $G$ is simply connected.
But when $G=\Spin(n)$, we will usually replace it by $\SO(n)$,
at the expense of possibly making $G/H$ not simply connected.
We fix an auxiliary biinvariant metric $\langle\cdot,\cdot\rangle_{\mathrm{bi}}$ with  norm $||\cdot||_{\mathrm{bi}}$ on the Lie algebra $\fg$ of $G$
and let $\fm$ denote the orthogonal complement of  $\fh\subset\fg$.
A homogeneous Finsler  metric is determined by a smooth  map $F\colon\fm \to \R$
such that $F(\lambda v)=\lambda  F(v)$ for $\lambda>0$, and such that the unit ball $B:=\{v\in\fp\mid F(v)=1\}$ is a strictly convex $\Ad_H$ invariant hypersurface surrounding the origin. We refer to \cite{BCS} for more details on Finsler metrics. The Finsler metric is called reversible if $F$ satisfies $F(-v)=-F(v)$ as well. In this and the following section we do not need $F$ to be  reversible.

Many classical theorems and techniques in Riemannian geometry can be similarly
carried out in Finsler geometry. Here we mention some relevant examples.
 The usual proof in the Riemannian case also shows here that the fixed point set $\mathrm{Fix}(L)$ of a group of isometries $L$ is totally geodesic submanifold. So when
$G/H$ admits positive curvature, so does $\mathrm{Fix}(L)$. In Finsler geometry, we also have the Bonnet Myer theorem and the Synge theorem whose proofs are almost the same as  in Riemannian geometry \cite{BCS}.
In \cite{Finslereven} it was shown that if $M^n=G/H$ admits a homogeneous Finsler metric with positive Flag curvature, then $\rank G =\rank H$ if $n$ is even, and $\rank G =\rank H+1$ if $n$ is odd. We remark that this also follows  immediately from the general fact that an isometric torus action on an even dimensional positively curved manifold has a fixed point and in odd dimension has a circle orbit, by applying it to a maximal torus of $G$. The claim about the torus action follows e.g. from Weinstein's theorem on fixed point sets of isometries, which carries over to the Finsler case with the same proof
\cite{Ko}.

We now shortly recall the inductive technique from \cite{WZ}, which will used in the next section.
Given a subgroup $L\subset\H\subset G$, let $C(L)$ be the identity component of the centralizer of $L$ in $G$. It is well known that
 $C(L)$ acts transitively on the component of the fixed point set $\mathrm{Fix}(L)\subset G/H$ through the base point $o=eH$.
Thus $\mathrm{Fix}(L)=C(L)/C(L)\cap H$ is a totally geodesic submanifold of $G/H$ and hence
has positive flag curvature.
The proof will be carried out  by induction on the dimension of the Lie group, that is,
at all times we will assume that the main theorem holds for all Lie groups with dimension strictly below $\dim G $. Thus, if one of these fixed point sets does not belong to the list of positively curved homogeneous spaces, or the exceptions in Theorem A, we will simply say its does not admit positive curvature by induction. Note that for  simplicity of notation, we do not distinguish between components of the  centralizer group or the fixed point set. Furthermore, $\mathrm{Fix}(L)$ may not be simply connected, but the inductive argument can then be applied to its universal cover.

For $L$ we mainly use subgroups generated by  $\iota\in H$ such that $\Ad_\iota$ is an involution  and write
   $\Hi:=C(\iota)\cap H$ and thus $\mathrm{Fix}(\iota)=C(\iota)/\Hi$ is positively curved.

    Since $\iota$ is contained in a maximal torus of $H$, which in turn can be extended to a maximal torus of
 $G$, we have $\rank C(\iota)=\rank G$ and  $\rank(\Hi)=\rank(H)$.
   Hence the codimension of these fixed point sets $C(\iota)/H^\iota$ is always even,
   and  we can do the induction in even and odd dimensions separately. Furthermore, $H$ is an equal rank extension of $\Hi$ which are fairly rare. Finally, the above condition on the ranks also implies that if $\H$ is regular in $\G$, i.e. all roots of $\fh$ are restrictions of roots of $\fg$, then $\Hi$ is regular in  $C(\iota)$ as well. Thus in each induction step we can again assume that all fixed point sets are odd dimensional with regular subalgebras.  In case of the classical Lie groups we can assume that, up to conjugacy, the maximal torus of $H$ is contained in the set of
diagonal matrices  and hence $\iota\in H$ will be diagonal as well.

   We will use the following facts frequently.
\begin{itemize}
\item If $G/H$ is a symmetric space of rank bigger than $1$, then it does not admit a homogeneous Finsler metric with positive flag curvature.
    \item If $G$ is not semi-simple, and $G/H$ admits a Finsler metric with positive flag curvature, then $G/H$ must be a homogeneous sphere
$\S^{2n-1}=\U(n)/U(n-1)$, $\S^{4n-1}=\Sp(n)\U(1)/\Sp(n)\U(1)$, or a
$\U(3)$-homogeneous Aloff-Wallach space
\item Assume that $G/H$ is almost effective and  admits a homogeneous Finsler metric with positive flag curvature. If $G$ is  semisimple but  not simple, or $G$ is simple but $H$ is
     not regular in $G$, then it also admits a Riemannian metric with positive sectional curvature.
     \item If $H$ is regular in $G$, and $G$ has two different simple factors acting nontrivially on $G/H$, then $G/H$ does not admit a reversible positively curved $G$-invariant Finsler metric.
         \item The homogeneous spaces $S^3\times S^3/S^1_{p,q}$, where $S^1_{p,q}=\{(e^{ip\theta},e^{iq\theta})\mid\theta\in\R\}$, and $\SO(5)/\SO(3)$ $=\Sp(2)/\SU(2)$, do not admit a reversible homogeneous Finsler metric with positive Flag curvature.
\end{itemize}
The first statement is an immediate corollary of the homogeneous flag curvature
formula (see Proposition \ref{flagcurv} below).
The remaining statements are  proved in \cite{Finslerodd}.

We can thus assume  for the fixed point sets $\Fi$ that either only one simple factor acts nontrivially, or it is one of the positively curved Riemannian homogeneous spaces with positive curvature with non-simple $G$.   We will thus frequently say "by induction" we can assume that we only need to consider the following homogeneous space.

Finally, we remark that regular subalgebras are easily classified. Indeed the maximal ones are determined by Borel-Siebenthal theory and a list is contained in \cite{Wo}.

\section{Simple Lie groups $G$}

In this section we assume that $G/H$ is odd dimensional and hence $\rank G=\rank H+1$, and that $G$ simple and $\fh$ regular in $\fg$. The proof is via induction over the dimension of $G$. We also use the notation $T_H, T_G$ for a maximal torus of $H$ and $G$ and can assume that $T_H\subset T_G$.
We now go through each simple Lie group one at a time.

\begin{thm} Let $M=G/H$ be an odd dimensional simply connected reversible homogeneous Finsler manifold with positive flag curvature, $G=\SU(n)$ and $H$ regular in $G$. Then either
$G/H=\SU(n)/\SU(n-1)=\S^{2n-1}$, $\SU(3)/\S^1_{(p,q,-p-q)}$ with $pq(p+q)\neq 0$, or possibly $G/H=\SU(4)/\diag(zA,z,\bar z^3)$ with $A\in\SU(2)$ and $z\in\C$.
\end{thm}
\begin{proof}
 For convenience, we use the notation of e.g.
$\SU(4)_{(1,2,3,4)}\subset\SU(n)$ for the block embedding of $\SU(4)$ in the first 4 coordinates, and $\S^1_{(p_1,p_2,\cdots,p_n)}\subset\SU(n)$ for the circle subgroup $\diag(z^{p_1},z^{p_2},\cdots,z^{p_n})$ of the diagonal torus of $\SU(n)$. We start with the case of small values of $n$.

\subsection{$G=\SU(3)$}
The only regular corank 1 subalgebras are $\SU(2)$ and $\S^1_{p,q}=\diag(z^p,z^q,\bar z^{p+q})$ with $\mathrm{gcd}(p,q)=1$.
In the first case, the quotient is a sphere, and in the second case it is an Aloff Wallach space. These admit a homogeneous metric with positive curvature except when $pq(p+q)=0$. In that case, say  $(p,q)=(1,0)$, we can
 choose the involution $\iota=\diag(-1,1,-1)\in H$ whose fixed point set $\Fi=\U(2)/H=(\Sph^2\times\Sph^1)/\Delta\Z_2$ cannot have positive curvature by Bonnet-Myers.

\subsection{$G=\SU(4)$} There exists an involution $\iota\in H$ which is not central in $G$ since $\rank(H)=2$. We may assume $\iota=\diag(1,1,-1,-1)$. Thus $\Fi=S(\U(2)\U(2))/\Hi$ has positive flag curvature and $\Hi$ is regular in $S(\U(2)\U(2))$  with rank 2, i.e. $\Hi=\SU(2)\cdot S^1$. Thus $\Hi$, and hence $H$, must contain a block $\SU(2)$. Up to conjugacy, this leaves only the possibilities $H=\SU(3)$,  $H=\SU(2)_{(1,2)}\SU(2)_{(3,4)}$ or $H=\SU(2)_{(1,2)}\cdot\diag(\bar z^{p+q},\bar z^{p+q},z^{2p},z^{2q})$
   with $\gcd\mathrm{gcd}(p,q)=1$. In the first case  the quotient is a sphere, and in the second case  we can choose the involution $\iota_2=\diag(i,-i,i,-i)\in  H$ with
   \begin{eqnarray*}
   \Fis=\S(\U(2)\cdot\U(2))/\S(\U(1)\cdot\U(1))\cdot\S(\U(1)\cdot\U(1))=\SU(2)\cdot\SU(2)/S^1
   \end{eqnarray*}
   which does not admit positive Flag curvature.

In the third case $H=\SU(2)_{(1,2)}\cdot\S^1_{(p+q,p+q,-2p,-2q)}$, and up to conjugacy we can assume that $\mathrm{gcd}(p,q)=1$, $p+q\ge 0$ and $p\geq q$. In Example 1 we will show that in the generic case, but under the assumption that $(p,q)\neq (1,0), (1,-1), (1,1) \mbox{ or }(3,-1)$, $G/H$ does not admit positive flag curvature. In Example 2 we will show that this also holds when $(p,q)=(1,1)$. But the special cases
$(p,q)=(1,-1)$ and $(p,q)=(1,0)$ can be excluded using the fixed point technique, leaving only $(p,q)=(3,-1)$ as a possible exception.

When $(p,q)=(1,-1)$, i.e.
$H=\SU(2)_{(1,1)}\cdot\S^1_{(0,0,1,-1)}$, the fixed point set
   for $L=\S^1_{(1,-1,1,-1)}\subset H$ is
   $\mathrm{Fix}(L)=\SU(2)_{(1,3)}\cdot\SU(2)_{(2,4)}/\S^1_{(1,-1,-1,1)}$. But it does not admit positive curvature.

When $(p,q)=(1,0)$, i.e. $H=\SU(2)_{(1,1)}\cdot\S^{1}_{(1,1,-2,0)}$,
the fixed point set of $\S^1_{(-2,1,1,0)}\subset H$ is
\begin{eqnarray*}
\Fix(\S^1_{(-2,1,1,0)})&=&
\SU(2)_{(2,3)}\cdot\S^1_{(-2,1,1,0)}\cdot\S^1_{(1,1,1,-3)}/
\S^1_{(1,-1,0,0)}\cdot\S^1_{(1,1,-2,0)}\\
&=&(\SU(2)_{(2,3)}/\S^1_{(0,1,-1,0)})\times\S^1_{(1,1,1,-3)},
\end{eqnarray*}
which does not admit positive curvature.

\bigskip

\subsection{$G=\SU(5)$}  One easily sees that since $H$ has rank 3, there must be an involution in $H$  which has an eigenvalue $-1$ of multiplicity 4. Thus we can assume that the involution $\iota=\diag(-1,-1,-1,-1,1)\in H$ and
 $\Fi=\U(4)/\Hi=\SU(4)_{(1,2,3,4)}\cdot\S^1_{(1,1,1,1,-4)}/\Hi$  must have positive flag curvature. By induction, and since $\Hi$ is a regular subalgebra,
there are 3 possible choices for $\Hi$:
$$
  \Hi=\SU(4),\quad \SU(3)_{(1,2,3)}\cdot\S^1_{(a,a,a,b,-3a-b)},\quad \SU(2)_{(1,2)}\cdot\S^1_{(1,1,1,-3,0)}\cdot\S^1_{(1,1,1,1,-4)}.
$$
Since $H$ is an equal rank extension of $\Hi$, it must be
\begin{eqnarray*}
  H&=&\SU(4)_{(1,2,3,4)},\quad\SU(3)_{(1,2,3)}\cdot\SU(2)_{(4,5)},\quad
  \SU(2)_{(1,2)}\cdot\SU(2)_{(4,5)}\cdot\S^1_{(2,2,2,-3,-3)},\\
  & &\SU(2)_{(1,2)}\cdot\S^1_{(1,1,1,-3,0)}\cdot\S^1_{(1,1,1,1,-4)},\
 \text{ or } \SU(3)_{(1,2,3)}\cdot\S^1_{(a,a,a,b,-3a-b)}.
\end{eqnarray*}

If $H=\SU(4)_{(1,2,3,4)}$, the quotient is a sphere.
If $H=\SU(3)_{(1,2,3)}\cdot\SU(2)_{(4,5)}$, $\SU(2)_{(1,2)}\cdot\SU(2)_{(4,5)}\cdot\S^1_{(2,2,2,-3,-3)}$,
or $\SU(2)_{(1,2)}\cdot\S^1_{(1,1,1,-3,0)}\cdot\S^1_{(1,1,1,1,-4)}$,
we can choose the involution $\iota_2=\diag(-1,-1,1,-1,-1)\in H$. Then
\begin{eqnarray*}
\mathrm{Fix}({\iota_2})&=&\SU(4)_{(1,2,4,5)}\cdot\S^1_{(1,1,-4,1,1)}/
\SU(2)_{(1,2)}\cdot\SU_{(4,5)}\cdot\S^1_{(1,1,-2,0,0)},\\
& \text{ or } &\SU(4)_{(1,2,4,5)}\cdot\S^1_{(1,1,-4,1,1)}/
\SU(2)_{(1,2)}\cdot\SU(2)_{(4,5)}\cdot\S^1_{(2,2,2,-3,-3)},\\
& \text{ or } &\SU(4)_{(1,2,4,5)}\cdot\S^1_{(1,1,-4,1,1)}/
\SU(2)_{(1,2)}\cdot\S^1_{(1,1,1,-3,0)}\cdot\S^1_{(1,1,1,1,-4)}
\end{eqnarray*}
respectively. Notice the abelian factor $\S^1_{(1,1,-4,1,1)}\subset C(\iota_2)$ acts nontrivially on these three spaces. But they do not
admit positively curved  metrics.

If $H=\SU(3)_{(1,2,3)}\cdot\S^1_{(a,a,a,b,-3a-b)}$, we can choose the
integers $a$ and $b$ such that $a\geq 0$ and $\mathrm{gcd}(a,b)=1$.
When $a$ is even and $b$ is odd,
we can take the involution $\iota_2=\diag(-1,-1,1,-1,-1)$. If $3a+2b\neq 0$,
$$\mathrm{Fix}(\iota_2)=\SU(4)_{(1,2,4,5)}\cdot\S^1_{(1,1,-4,1,1)}/
\SU(2)_{(1,2)}\cdot\S^1_{(1,1,-2,0,0)}\cdot\S^1_{(a,a,a,b,-3a-b)},$$ where
$\S^1_{(1,1,-4,1,1)}$ acts nontrivially, but it does not admit positively
curved homogeneous Riemannian metrics. If $3a+2b=0$,
\begin{eqnarray*}
\mathrm{Fix}(\iota_2)&=&\SU(4)_{(1,2,4,5)}\cdot\S^1_{(1,1,-4,1,1)}/
\SU(2)_{(1,2)}\cdot\S^1_{(1,1,-2,0,0)}\cdot\S^1_{(2,2,2,-3,-3)}\\
&=&\SU(4)_{(1,2,4,5)}/\SU(2)_{(1,2)}\S^1_{(1,1,0,-1,-1)}
\end{eqnarray*}
which does not admit positive curvature by induction.

When $a$ is odd and
$b$ is even, we can take the involution $\iota_2=\diag(1,1,-1,1,-1)\in H$.
Then $\mathrm{Fix}(\iota_2)=\SU(3)_{(1,2,4)}\cdot
\SU(2)_{3,5}\cdot\S^1_{(2,2,-3,2,-3)}/\SU(2)_{(1,2)}\cdot\S^1_{(1,1,-2,0,0)}
\cdot\S^1_{(a,a,a,b,-3a-b)}$. Both simple factors in $C(\iota_2)=\SU(3)_{(1,2,4)}\cdot
\SU(2)_{(3,5)}\cdot\S^1_{(2,2,-3,2,-3)}$ acts
nontrivially but it is not the Wilking space. So it does not admit positive curvature.
When both $a$ and $b$ are odd, we can take the involution
$\iota_2=\diag(1,1,-1,-1,1)$, and the argument is similar.

 \subsection{$G=\SU(6)$} A non-central involution $\iota\in H$ has an eigenvalue $-1$ with multiplicity 2 or 4, and hence up to conjugacy, has the form $\iota=\pm\diag(-1,-1,-1,-1,1,1)$ with fixed point set
$\Fi=\SU(4)_{(1,2,3,4)}\cdot \SU(2)_{(5,6)}\cdot\S^1_{(1,1,1,1,-2,-2)}/\Hi$. By induction, the
 possible choices for $\Hi$ are:
  \begin{eqnarray*}
  \Hi&=&\SU(4)_{(1,2,3,4)}\cdot\S^1_{(a,a,a,a,b,-4a-b)},\quad \SU(4)_{(1,2,3,4)}\cdot\SU(2)_{(5,6)},\\
  & \text{ or } &\SU(3)_{(1,2,3)}\cdot\SU(2)_{(5,6)}\cdot\S^1_{(a,a,a,-3a-2c,c,c)},\\
  &\text{ or } &\SU(2)_{(1,2)}\cdot\SU(2)_{(5,6)}\cdot\S^1_{(1,1,1,-3,0,0)}
  \cdot\S^1_{(1,1,1,1,-2,-2)}.
  \end{eqnarray*}
Since $H$ is an equal rank extension of $\Hi$, the only possible choices for $H$ are
\begin{eqnarray*}
H&=&\SU(5),\quad \SU(4)_{(1,2,3,4)}\cdot\SU(2)_{(5,6)},\quad
\SU(3)_{(1,2,3)}\cdot\SU(3)_{(4,5,6)},\\
& &\SU(4)_{(1,2,3,4)}\cdot\S^1_{(a,a,a,a,b,-4a-b)},\quad
\SU(3)_{(1,2,3)}\cdot\SU(2)_{(5,6)}\cdot\S^1_{(a,a,a,-3a-2c,c,c)},\\
&\text{ or } &\SU(2)_{(1,2)}\cdot\SU(2)_{(5,6)}\cdot\S^1_{(1,1,1,-3,0,0)}
  \cdot\S^1_{(1,1,1,1,-2,-2)}.
\end{eqnarray*}

If $H=\SU(5)$, then $G/H$ is a homogeneous sphere.
If $H=\SU(4)_{(1,2,3,4)}\cdot\SU(2)_{(5,6)}$ or $\SU(3)_{(1,2,3)}\cdot\SU(3)_{(4,5,6)}$, we can take the
involution $\iota_2=\diag(1,-1,-1,-1,-1,1)\in H$ with the fixed point set
$\mathrm{Fix}(\iota_2)=C(\iota_2)/\His$, where
$C(\iota_2)=\SU(4)_{(2,3,4,5)}\cdot\SU(2)_{(1,6)}\cdot\S^1_{(-2,1,1,1,1,-2)}$
and $\His=\SU(3)_{(2,3,4)}\cdot\S^1_{(-3,1,1,1)}\cdot\S^1_{(0,0,0,0,1,-1)}$
or $\SU(2)_{(2,3)}\cdot\SU(2)_{(4,5)}\cdot\S^1_{(-2,1,1,0,0,0)}
\cdot\S^1_{(0,0,0,1,1,-2)}$ respectively. Notice each simple factor of
$C(\iota_2)$ acts nontrivially, but the fixed point sets do not admit positively curved homogeneous Riemannian metrics.

If $H=\SU(2)_{(1,2)}\cdot\SU(2)_{(5,6)}\cdot\S^1_{(1,1,1,-3,0,0)}\cdot
S^1_{(1,1,1,1,-2,-2)}$, we can choose the involution $\iota_2=
\diag(1,1,-1,-1,1,1)$. The fixed point set
$\mathrm{Fix}(\iota_2)=\SU(4)_{(1,2,5,6)}\SU(2)_{(3,4)}\S^1_{(1,1,-2,-2,1,1)}
/H$, where both simple factors of $C(\iota_2)$ act nontrivially.
But it does not admit positive curvature.

If $H=\SU(3)_{(1,2,3)}\cdot\SU(2)_{(5,6)}\cdot\S^1_{(a,a,a,-3a-2c,c,c)}$
and $a= -c$, i.e. $G/H=\SU(6)/\SU(3)_{(1,2,3)}\cdot\SU_{(5,6)}\cdot
\S^1_{(1,1,1,-1,-1,-1)}$, we can choose the involution
$\iota_2=\diag(1,-1,-1,1,-1,-1)\in H$. Then the fixed point set
$$\mathrm{Fix}(\iota_2)=\SU(4)_{(2,3,5,6)}\cdot\SU(2)_{(1,4)}
\cdot\S^1/\SU(2)_{(2,3)}\cdot\S^1\cdot
\S^1\cdot\S^1,$$
where both simple factors in $C(\iota_2)$ acts nontrivially.
But it does not admit positive curvature.

For the remaining cases, we can choose the involution
$\iota_2=\diag(-1,-1,1,1,1,1)\in H$. If $H=\SU(4)_{(1,2,3,4)}\cdot\S^1_{(a,a,a,a,b,-4a-b)}$ and $b=-2a$,
the fixed point set
$$\SU(4)_{(3,4,5,6)}\cdot\SU(2)_{(1,2)}\cdot\S^1_{(-2,-2,1,1,1,1)}/
\SU(2)_{(1,2)}\cdot\SU(2)_{(3,4)}\cdot\S^1_{(1,1,-1,-1,0,0)}\cdot
\S^1_{(1,1,1,1,-2,-2)}$$
is isometric to $\SU(4)/\SU(2)_{(1,2)}\cdot\S^1_{(1,1,-1,-1)}$,
 which does not admit positive
curvature by induction. If $H=\SU(4)_{(1,2,3,4)}\cdot\S^1_{(a,a,a,a,b,-4a-b)}$ and $b\neq-2a$,
the fixed point set
$
\mathrm{Fix}(\iota_2)
=\SU(4)_{(3,4,5,6)}\cdot\S^1_{(-2,-2,1,1,1,1)}/
\SU(2)_{(3,4)}\cdot\S^1_{(1,1,-1,-1,0,0)}\cdot
\S^1_{(a,a,a,a,b,-4a-b)}
$ where the Abelian factor $\S^1_{(-2,-2,1,1,1,1)}$ acts nontrivially.
But it does not admit positive curvature.

If $H=\SU(3)_{(1,2,3)}\cdot\SU(2)_{(5,6)}\cdot\S^1_{(a,a,a,-3a-2c,c,c)}$
and $a\neq -c$,
the fixed point set
$
\mathrm{Fix}(\iota_2)=
\SU(4)_{(3,4,5,6)}\cdot\S^1_{(-2,-2,1,1,1,1)}/
\SU(2)_{(5,6)}\cdot\S^1_{(1,1,-2,0,0,0)}\cdot\S^1_{(a,a,a,-3a-2c,c,c)}
$, where the Abelian factor $\S^1_{(-2,-2,1,1,1,1)}$ acts nontrivially.
But it does not admit positively curved Riemannian homogeneous
metrics.

 \subsection{$G=\SU(n), n\ge 7$}
 Since $T_H$ has  codimension 1 in $T_G$, the 2 dimensional torus in a block $\SU(3)\subset\SU(n)$ has a 1-dimensional intersection with $T_H$. Thus we can assume that the involution
$\iota=(-1,-1,1,\dots,1)$ lies in $T_H$. But then $\Fi=\SU(2)\cdot\SU(n-2)\cdot\S^1/\Hi$ has positive flag curvature. Since $n-2\ge 5$, by induction $\Hi$ must be one of
$$
\Hi=\SU(2)_{(1,2)}\cdot\SU(n-2)_{(3,\cdots,n)},\quad \SU(n-2)_{(3,\cdots,n)}\cdot\S^1, \mbox{ or } \SU(2)_{(1,2)}\cdot\SU(n-3)_{(3,\cdots,n)}\cdot\S^1.
$$
Since $H$ is an equal rank enlargement of $\Hi$,  it is up to conjugation either $\SU(n-1)$, in which case $G/H$ is a sphere, or  one of
\begin{eqnarray*}
H&=&\SU(2)_{(1,2)}\cdot\SU(n-2)_{(3,\cdots,n)},\quad \SU(n-2)_{(3,\cdots,n)}\S^1,\\
& &\SU(2)_{(1,2)}\cdot\SU(n-3)_{(3,\cdots,n-1)}\cdot\S^1, \  \text{ or } \SU(3)_{(1,2,3)}\cdot\SU(n-3)_{(4,\cdots,n)}.
\end{eqnarray*}
Notice all abelian factors $\S^1$ above and below may be quite arbitrary and different from each other.

If $H=\SU(3)_{(1,2,3)}\cdot\SU(n-3)_{(4,\cdots,n)}$,
we can take the involution
$$\iota_2=\diag(1,-1,-1,-1,-1,1,\cdots,1)\in H.$$
Then the fixed point set $\mathrm{Fix}(\iota_2)$ is
$$\SU(4)_{(2,3,4,5)}\cdot\SU(n-4)_{(1,6,\cdots,n)}
\cdot\S^1/
\SU(2)_{(2,3)}\cdot\SU(2)_{(4,5)}\cdot\SU(n-5)_{(6,\cdots,n)}\cdot\S^1\cdot\S^1,$$
 where both simple factors in $C(\iota_2)$ act nontrivially and hence does
  not admit positively curved  metrics.

For the other cases, we can take the
involution $\iota_2=\diag(1,1,-1,-1,1,\cdots,1)$ from $H$. Then the fixed point set
$\mathrm{Fix}(\iota_2)$ is isometric to
$\SU(n-2)\cdot\S^1/\SU(2)\cdot
\SU(n-4)\cdot\S^1$,
$\SU(n-2)\S^1/\SU(n-4)\cdot\S^1\cdot\S^1$ or
$\SU(n-2)\cdot\S^1/\SU(2)_{(1,2)}\cdot\SU(n-5)\cdot\S^1\cdot\S^1$
respectively. But it does not admit positive curvature, both in the case where
 the abelian factor in $C(\iota_2)$ acts trivially or nontrivially

This finishes the proof in the case $G=\SU(n)$.
\end{proof}

\begin{thm} Let $M=G/H$ be an odd dimensional simply connected reversible homogeneous Finsler manifold with positive flag curvature, $G=\Sp(n)$ and $H$ regular in $G$. Then either
$G/H=\Sp(n)/\Sp(n-1)=\S^{4n-1}$, or possibly $\Sp(2)/\diag(z,z),\ \Sp(2)/\diag(z,z^3)$ or  $\Sp(3)/\diag(z,z,r)$, with $z\in\C,\ q\in\QH$.
\end{thm}
\begin{proof}
The proof is by induction over $\dim G$.

 \subsection{$G=\Sp(2)$}

 The only rank one regular subalgebras are a block $\Sp(1)$, whose quotient is a sphere,  $H=\SU(2)$,  or $H=\diag(z^p,z^q)$ with $z\in\C$.  When $H=\SU(2)$, it was proven in \cite{Finslerodd} that it does not admit positive curvature.

 When $H=\diag(z^p,z^q)$, by suitable conjugations, we can assume $p\geq q\geq 0$, and $\mathrm{gcd}(p,q)=1$. If $q=0$, the fixed point set
 for the unique involution in $H$ is $(\Sp(1)/S^1)\times\Sp(1)$, which does not
 admit positive curvature. In Example 3 we will  show that when
 $(p,q)\neq(1,1)$ or $(1,3)$, it does not admit positive curvature either.

 \subsection{$G=\Sp(3)$} First, observe that for all $n\ge 3$  the 2 dimensional torus in $\SU(3)\subset\Sp(3)\subset \Sp(n)$ has a 1-dimensional intersection with $T_H$, and hence we  can assume that
$$\iota=(-1,-1,1,\dots,1)\in H.$$
   So for $n=3$, we have a fixed point set $\Fi=\Sp(2)_{(1,2)}\cdot\Sp(1)_{(3)}/\Hi$ which has positive flag curvature.
 We can also assume that $H$ is not equal to $\Sp(2)_{(1,2)}$ since otherwise $G/H$ is a sphere.
Thus $\Sp(2)_{(1,2)}$ acts effectively on $\Fi$ and hence by induction, we only need to consider the following cases.

In the first case $H=\Hi=\Sp(1)_{(2)}\cdot\Sp(1)_{(3)}$ or $\Sp(1)_{(1)}\cdot\Sp(1)_{(3)}$. But this is not allowed since $\iota\notin \Hi$.
In the second case, $H=\Hi=\Sp(1)_{(3)}\cdot\S^1_{p,q,0}$ with $(p,q)=(1,1)$
or $(1,3)$. In  Example 4 we will show that in the case of $(p,q)=(1,3)$, $G/H$ does not admit
positive curvature.

 \subsection{$G=\Sp(4)$}
Here we have a fixed point set $\Fi=\Sp(2)\cdot\Sp(2)/\Hi$ which is totally geodesic. By induction we have, up to conjugacy, that $\Hi=\Sp(2)_{(1,2)}\cdot\Sp(1)_{(3)}$, or $\Hi=\Sp(2)_{(1,2)}\cdot\S^1_{(0,0,1,p)}$ with $p=1$ or $3$.
This leaves, up to conjugacy, the following possibilities for $H$, apart from $H=\Sp(3)$ where $G/H$ is a sphere:
$$
H=\Sp(2)_{(1,2)}\cdot\Sp(1)_{(3)},\quad
\Sp(2)_{(1,2)}\cdot\SU(2)_{(3,4)},\quad \Sp(2)_{(1,2)}\cdot\S^1_{(0,0,1,p)}.
$$
In the first two cases, we can choose the involution $\iota_2=\diag(1,-1,-1,-1)\in H$
 with fixed point set $$\mathrm{Fix}(\iota_2)=\Sp(3)_{(2,3,4)}\cdot\Sp(1)_{(1)}/
 \Sp(1)_{(1)}\cdot\Sp(1)_{(2)}\cdot\Sp(1)_{(3)}=
 \Sp(3)_{(2,3,4)}/
 \Sp(1)_{(2)}\cdot\Sp(1)_{(3)}$$
 in the first case and
 $$\mathrm{Fix}(\iota_2)=\Sp(3)_{(2,3,4)}\cdot\Sp(1)_{(1)}
 / \Sp(1)_{(1)}\cdot\Sp(1)_{(2)}\cdot\SU(2)_{(3,4)}
 =\Sp(3)_{(2,3,4)}
 / \Sp(1)_{(2)}\cdot\SU(2)_{(3,4)}$$ in the second case,
 neither one of which admits positive curvature by induction.

When $H=\Sp(2)_{(1,2)}\cdot\S^1_{(0,0,1,3)}$, the fixed point
set for $L=\Sp(1)_{(1)}\subset H$ is $\Fix(L)=\Sp(3)/\Sp(1)_{(1)}\cdot\S^1_{(0,1,3)}$
which does not admit positive curvature by induction.
When $H=\Sp(2)_{(1,2)}\cdot\S^1_{(0,0,1,1)}$, the fixed point
set for $L=\S^1_{(1,1,1,1)}\subset H$ is
$$\Fix(L)=\SU(4)\cdot\S^1_{(1,1,1,1)}/\SU(2)_{(1,2)}\cdot\S^1_{(1,1,0,0)}\S^1_{(0,0,1,1)}
=\SU(4)/\SU(2)_{(1,2)}\cdot\S^1_{(1,1,-1,-1)}$$
which does not admit positive curvature by Example 2.

\bigskip

 \subsection{$G=\Sp(n),n\ge 5$}
As before, we can use the involution $\iota=\diag(-1,-1,1,\cdots,1)\in H$ and conclude that $\Fi=\Sp(2)\cdot\Sp(n-2)/\Hi$ has positive flag curvature. By induction $\Hi$ must be, up to conjugacy, one of
$$
\Sp(1)_{(1)}\cdot\Sp(n-2)_{(3,\cdots,n)},\quad \Sp(2)_{(1,2)}\cdot\Sp(n-3)_{(3,\cdots,n-1)},
\mbox{ and } \S^1_{(1,p,0,\cdots,0)}\cdot\Sp(n-2)_{(3,\cdots,n)}
$$
with $p=1$ or $3$, and in the case of $n=5$ we need to add one more possibility, namely, $$ \Hi=\Sp(2)_{(1,2)}\cdot\Sp(1)_{(3)}\cdot\S^1_{(0,0,0,1,1 )}. $$

We can now assume that $H=\Hi$ since we can exclude the only further equal rank extension $H=\SU(2)_{(1,2)}\cdot\Sp(n-2)_{(3,4,\cdots,n)}$ since then $\Fi=\Sp(2)/\SU(2)$ is not positively curved.

In the first 2 cases we choose $\iota_2=\diag(-1,-1,-1,1,\cdots,1)\in H$. Then
  the fixed point set is isometric to $$\mathrm{Fix}(\iota_2)=\Sp(3)_{(1,2,3)}\cdot\Sp(n-3)_{(4,\cdots,n)}
  /\Sp(1)_{(1)}\cdot\Sp(2)_{(2,3)}\cdot\Sp(n-4)_{( 4,\cdots,n-1 )}$$
 which by induction does not positive curvature.
 In the exceptional case when $n=5$ and $$H= \Sp(2)_{(1,2)}\cdot\Sp(1)_{(3)}\cdot\S^1_{(0,0,0,1,1 )},$$ we can take
 the same involution $\iota_2=\diag(-1,-1,-1,1,1)\in H$ as above. But then
 $$\Fis=(\Sp(3)_{(1,2,3)}/\Sp(2)_{(1,2)}\cdot\Sp(1)_{(3)})\times(\Sp(2)_{(4,5)}
 /\cdot\S^1_{(0,0,0,1,1 )}) $$ does not admit positive curvature either.

In the last case $H=\Hi=\S^1_{(1,p,0,\cdots,0)}\cdot\Sp(n-2)_{(3,4,\cdots,n)}$, we can choose the involution $\iota_2=\diag(1,1,-1,1,\cdots,1)\in H$
 with fixed point set
 \begin{eqnarray*}
 \mathrm{Fix}(\iota_2)&=&\Sp(1)_{(3)}\cdot
 \Sp(n-1)_{(1,2,4,\cdots,n)}/\S^1_{(1,p,0,\cdots,0)}\cdot
 \Sp(1)_{(3)}\Sp(n-3)_{( 4,\cdots,n )}\\
 &=&\Sp(n-1)_{(1,2,4,\cdots,n)}/\S^1_{(1,p,0,\cdots,0)}\cdot\Sp(n-3)_{( 4,\cdots,n )},
 \end{eqnarray*}
 which by induction does not admit positive curvature since $n-1\ge 4$.
 \end{proof}

For the groups $G=\SO(n)$ or $\Spin(n)$ below, we can assume, due to low dimensional isomorphisms, that $n\ge 7$.

\begin{thm} Let $M=G/H$ be an odd dimensional simply connected reversible
homogeneous Finsler manifold with $G=\Spin(n)$ and $H$ regular. If $n\ge 7$, then $G/H$ does not admit positive curvature.
\end{thm}
\begin{proof}
As explained before, we can assume that $G=SO(n)$, and we will also assume $n\ge 7$.
The 2 dimensional torus in $\SU(3)\subset\SO(6)$  contained in the upper the $6\times 6$-block has a 1-dimensional intersection with $T_H$ and any involution
$\iota$ in this intersection has the eigenvalue $-1$ with  real multiplicity four. Hence without
loss of generality we can choose
$\iota=\diag(-1,-1,-1,-1,1,\cdots,1)\in H$.

 \subsection{$G=\SO(7)$}
The fixed point set  $\Fi=\SO(4)\cdot\SO(3)/\Hi$ has positive curvature. Only one simple factor in $C(\iota)=\SO(4)\cdot\SO(3)$ can act nontrivially on $\Fi$, and $\Hi$ is regular in $C(\iota)$. Hence we are left with 2 cases: $\Hi=\SO(4)$ or
$\SU(2)_\pm\cdot\SO(3)_{(5,6,7)}$ where $\SU(2)_\pm$ is one of the two nontrivial proper normal subgroups of $\SO(4)_{(1,2,3,4)}$. This leaves only the following possibilities: $H=\SO(4),\quad \SO(5),\quad \SU(2)_\pm\cdot\SO(3)$.

If $H=\SO(4)_{(1,2,3,4)}$, we consider the totally geodesic fixed point set for $L=\SO(3)_{(1,2,3)}\subset\SO(4)_{(1,2,3,4)}$ is
$$\mathrm{Fix}(L)=\SO(3)\cdot\SO(4)/\SO(3)=\SO(4),$$ which does not admit positive curvature.

If $H=\SO(5)_{(1,2,3,4,5)}$, we choose the involution $\iota_2=\diag(-1,-1,1,\cdots,1)\in H$
 with fixed point set $$\mathrm{Fix}(\iota_2)=\SO(2)\cdot\SO(5)
 /\SO(2)\cdot\SO(3)
 =\SO(5)/\SO(3)$$ which does not admit positive curvature.

If $H=\SU(2)_\pm\cdot\SO(3)_{(5,6,7)}$, we take the involution $\iota_2=\diag(-1,\cdots,-1,1)\in H$
 with fixed point set $\mathrm{Fix}(\iota_2)=\SO(6)/\SU(2)_\pm\cdot\SO(2)$. Observe that under the 2-fold cover $\SU(4)\to\SO(6)$ the subgroup $\S(\U(2)\cdot\U(2))\subset\SU(4)$ goes into $\SO(4)\cdot\SO(2)\subset \SO(6)$. Thus $\SU(2)_\pm\cdot\SO(2)$, when lifted to $\SU(4)$,
 becomes $\SU(2)_{(1,2)}\cdot\S^1_{(1,1,-1,-1 )}$. It implies that
$$\mathrm{Fix}(\iota_2)=\SO(6)/\SU(2)_{\pm}\cdot\SO(2)_{(5,6)}=
\SU(4)/\SU(2)_{(1,2)}\cdot\S^1_{(1,1,-1,-1 )}$$ which does not admit positive curvature by Example 2.

 \subsection{$G=\SO(8)$}
 The fixed point set  $\Fi=\SO(4)\cdot\SO(4)/\Hi$ has positive curvature, which implies that, up to conjugacy, $\Hi=\SU(2)_\pm\SO(4)_{(5,6,7,8)}=H$. Taking the
 involution
 $$\iota_2=\diag(1,1,1,1,1,-1,-1)\in H$$ we have the fixed point set
$$\Fis=\SO(6)\cdot\SO(2)/\SU(2)_\pm\cdot\SO(2)\cdot\SO(2)=\SO(6)/\SU(2)_\pm\cdot
\SO(2)_{(5,6)},$$ which we  saw does not admit positive curvature.

 \subsection{$G=\SO(n), n\ge 9$} Here $\Fi=\SO(4)\cdot\SO(n-4)/\Hi$ has positive curvature. If $H=\Hi=\SU(2)_\pm\cdot\SO(n-4)_{(5,\cdots,n)}$, we can choose the same involution $\iota_2=\diag(1,\cdots,1,-1,-1)\in H$ as in the previous case with $$\Fis=\SO(n-2)\cdot\SO(2)/\SU(2)_\pm\cdot\SO(n-6)\cdot\SO(2)
 =\SO(n-2)/\SU(2)_\pm\cdot\SO(n-6)$$ which by induction does not admit positive curvature since $n-2\ge 7$.

Now we consider the other cases where $\SO(n-4)_{(5,\cdots,n)}$ acts nontrivially on $\Fi$.
By induction, this may only happen when $n=9$ or $10$, where we have several possibilities for a stabalizer group $\bar H\subset\SO(n-4)_{(5,\cdots,n)}$ with $H=\SO(4)_{(1,2,3,4)}\cdot\bar H$. Again, we can take the involution $\iota_2=\diag(-1,-1,1,\cdots,1)\in H$ with $$\Fis=\SO(2)\cdot\SO(n-2)/\SO(2)\cdot\SO(2)\cdot\bar H= \SO(n-2)/\SO(2)\cdot\bar H $$ which does not admit positive curvature by induction.
\end{proof}

In the case of exceptional simple Lie groups $G$, we do not need to assume that $H$ is regular, or that $G/H$ is odd dimensional.

\begin{thm} Let $M=G/H$ be a simply connected  homogeneous Finsler manifold with positive flag curvature, with $G$ a compact exceptional simple Lie group. Then the only possible choices for $G/H$ are the Cayley plane $F_4/\Spin(9)$, the flag manifold $F_4/\Spin(8)$, or possibly
$\G_2/\SU(2)_\pm$, where $\SU(2)_\pm$ are the two normal subgroup $\SU(2)_\pm\subset\SO(4)\subset \G_2$.  In  the case of the normal subgroup corresponding to the short root, $G/H$ does not admit a reversible Finsler metric with positive flag curvature.
\end{thm}
\begin{proof}
Most arguments in \cite{WZ}  only uses the fixed point set technique, the fact that a symmetric space of rank $>1$ does not admit positive curvature,
and that a product homogeneous space $G_1/H_1 \times G_2/H_2$ cannot admit positive curvature either. In \cite{Finslerodd} it was shown that these also hold in the Finsler setting, without assuming reversability.  This in particular finishes the cases $G=F_4, E_6$ and $E_7$.

In the case of $G=G_2$, the argument proves that  $H=\SU(2)_\pm\subset\SO(4)\subset G_2$ where the roots of $\mathfrak{h}$ can be either long or short.  We will show in Example 5 that the roots of $\mathfrak{h}$ may only be the long ones.

In the case of $G=E_8$, the argument needs to be changed in the case of $H=E_7$. Choose an involution $\iota\in E_7$ such that $\Hi=E_6\cdot\S^1$.  But the centralizers of involutions in $E_8$ are one of  $C(\iota)=E_7\cdot\SU(2)$ or $\SO'(16)$ and hence the quotient $C(\iota)/\Hi$ does not admit positive curvature by induction.

\end{proof}

\bigskip

\section{Examples}

In this section we will discuss the following homogeneous spaces, which were left over in the inductive proof in Section 2, using the methods in \cite{Finslerodd}.

Let $G/H$ be a homogeneous spaces and $\fg=\fh+\fm$  an $\mathrm{Ad}(H)$-invariant
decomposition. The Finsler metric invariant under $G$ is identified with a norm $F$ defined on $\fm$. For any vector $u\in\fm$, we can define a Riemannian inner product $\langle w_1,w_2\rangle_u^F=(D^2F)_u(w_1,w_2)$ as the Hessian of $F$ at the nonzero tangent vector $u$. Sometimes we simply denote it as $g_u^F$. The reversibility of the Finsler metric implies $\langle x,y\rangle_u^F=\langle x,y\rangle_{-u}^F$.

The following proposition from \cite{Finslereven} provides a simple homogeneous flag curvature
formula. It can be either proved by Finslerian submersion technique, or deduced from a more general and complicated homogeneous flag curvature formula of L. Huang \cite{Huang2013}.

\begin{prop} \label{flagcurv}
Let $(G/H,F)$ be a  homogeneous Finsler manifold, and $\fg=\fh+\fm$ be an $\mathrm{Ad}(H)$-invariant
decomposition for $G/H$. Then for any linearly independent vectors
 $u,v\in \mathfrak{m}$ with $[u,v]=0$ and $\langle[u,\fm]_\fm,u\rangle_u^F=0$, the flag curvature for $u$ and $u\wedge v$ is given by:
\begin{equation*}
K^F(u,u\wedge v)=\frac{\langle U(u,v),U(u,v)\rangle_u^F}
{\langle u,u\rangle_u^F \langle v,v\rangle_u^F-
{\langle u,v\rangle_u^F}\langle u,v\rangle_u^F},
\end{equation*}
where $U\colon \mathfrak{m}\times \mathfrak{m}\to \mathfrak{m}$ is defined by
\begin{equation*}
\langle U(u,v),w\rangle_u^F=\frac{1}{2}(\langle[w,u]_\mathfrak{m},v\rangle_u^F
+\langle[w,v]_\mathfrak{m},u\rangle_u^F), \mbox{ for any }w\in\mathfrak{m}.
\end{equation*}
\end{prop}

\bigskip

We will use this result to produce 0 curvature planes by choosing linearly independent commuting vectors $u$ and $v$ with
\begin{equation}\label{zero}
\langle[u,\fm],u\rangle_u^F=\langle[u,\fm],v\rangle_u^F=
\langle[v,\fm],u\rangle_u^F=0
\end{equation}
and hence $U(u,v)=0$.
The arguments are similar to the ones in \cite{Finslerodd}.

\subsection{Example 1}
$G=\SU(4)$ and $H=\SU(2)_{(1,2)}\cdot\S^1_{( p+q, p+q,-2p,-2q)}$ with $\mathrm{gcd}(p,q)=1$,
$p+q>0$, $p\ge q$, $(p,q)\ne (1,1)$, $(1,-1)$ or $(3,-1)$.

We can choose any nonzero vectors $u\in\fg_{\pm(e_1-e_3)}$ and
$v\in\fg_{\pm(e_2-e_4)}$, as two linearly independent commuting vectors.

To prove (\ref{zero}), we need the following $g_u^F$-orthogonality and bracket properties.

Denote $\fm_0=\ft\cap+\fg_{\pm(e_1-e_3)}$, $\fm_1=\fg_{\pm(e_2-e_4)}$, and
$\fm_2$ is the sum of other root planes in $\fm$. Then $\fm=\fm_0+\fm_1+\fm_2$ provide a decomposition of $\fm$ preserved by the $\mathrm{Ad}$-actions of
$L=S^1_{(-p,2p+q,-p,-q)}\in T_H$. The conditions on $p$ and $q$ implies $\fm_0$ is the $\mathrm{Ad}(L)$-invariant subspace, and $\fm_1$ corresponds to
a unique irreducible representation of $L$ in $\fm$. Because $L$ preserves
$u$ as well as $F$, thus it preserves $g_u^F$ as well. So the above decomposition is $g_u^F$-orthogonal.

When restricted to $\fg_{\pm(e_1-e_3)}$, the $\mathrm{Ad}(T_H)$-invariance of $F$ implies that it coincide with the biinvariant inner product up to a scalar. So we have $$\langle u,[u,\ft\cap\fm]_\fm\rangle_u^F=
\langle u,[u,\ft]\rangle_u^F=0.$$

We can find $g\in \S^1_{(1,-1,0,0)}\subset T_H$ such that
$\mathrm{Ad}(g)|_{\fg_{\pm(e_1-e_3)}}=-\mathrm{Id}$. Obviously it preserves
$\ft\cap\fm$. So for any $w\in\ft\cap\fm$, we
have
$$\langle u,w\rangle_u^F=\langle\mathrm{Ad}(g)u,\mathrm{Ad}(g)w\rangle_{\mathrm{Ad}(g)u}
=\langle -u,w\rangle_{-u}^F=-\langle u,w\rangle_u^F,$$
i.e. $\langle u,\ft\cap\fm\rangle_u^F=0$.

 If we denote by $\fm'=\ft\cap\fm+[u,\ft]+\fm_2$, then the above argument shows that $\fm'$ is $g_u^F$-orthogonal to both $u$ and $v$. A calculation shows that $[u,\fm']_\fm\subset\fm'$
and $[v,\fm']_\fm\subset\fm'$. Then (\ref{zero}) follows immediately, and
$G/H$ does not admit positive curvature by Proposition \ref{flagcurv}.

\bigskip

\subsection{Example 2}
$G=\SU(4)$ and $H=\SU(2)_{(1,2)}\cdot\S^1_{(1,1,-1,-1)}$.

Denote by
\begin{eqnarray*}
\fm_0&=&\ft\cap\fm+\fg_{\pm(e_3-e_4)},\\
\fm_1&=&\fg_{\pm(e_1-e_3)}+\fg_{\pm(e_1-e_4)},\mbox{ and}\\
\fm_2&=&\fg_{\pm(e_2-e_3)}+\fg_{\pm(e_2-e_4)},
\end{eqnarray*}
Notice $\fm_0$ is the centralizer of
$\fh$ in $\fm$, and it is the Lie algebra of $\SU(2)_{(3,4)}$.
By a suitable $\mathrm{Ad}(\SU(2)_{(3,4)})$-action on the metric, we can assume there is a $F$-unit vector $u\in \fg_{\pm(e_1-e_3)}\subset\fm_1$,
such that $||u||_{\mathrm{bi}}$ reaches the maximum among all $F$-unit
vectors in $\fm_1$. Choose any nonzero vector $v\in\fg_{\pm(e_2-e_4)}
\subset\fm_2$, then we get the linearly independent commuting pair $u$ and $v$,
for which we will verify (\ref{zero}).

By the reversibility of $F$, we can similarly prove
$\langle u,\fm_0\rangle_u^F=0$. The biinvariant orthogonal complement of
$\mathbb{R}u$ in $\fm_1$ is $[\ft,u]+\fg_{\pm(e_1-e_4)}$. Our special choice
for $u$ implies $\langle u,[\ft,u]+\fg_{\pm(e_1-e_4)}\rangle_u^F=0$.
The $\mathrm{Ad}(\S^1_{0,-2,1,1})\subset\mathrm{Ad}(T_H)$-actions preserve
$u$ as well as $F$, and thus $g_u^F$. According to this representation,
$\mathfrak{m}$ is $g_u^F$-orthogonally decomposed as the sum of $\fm_0+\fm_1$
and $\fm_2$. Thus
$$\fm'=\ft\cap\fm+\fg_{\pm(e_3-e_4)}+[\ft,u]+\fg_{\pm(e_1-e_4)}$$
is $g_u^F$-orthogonal to both $u$ and $v$.
A direct calculation shows that
$[u,\fm]_\fm\subset\fm'$ and $[v,\fm]_\fm\subset\fm'$, and hence
(\ref{zero}) holds. Thus $G/H$ does not admit positive curvature.

\bigskip

\subsection{Example 3}
$G=\Sp(2)$ and $H=\S^1_{(p,q)}=\diag(z^p,z^q)$ with $\mathrm{gcd}(p,q)=1$, $p>q>0$ and $(p,q)\neq(3,1)$.

We have a root plane decomposition
\begin{eqnarray}\label{root-plane-decomp-1}
\fm=\ft\cap\fm+\fg_{\pm2e_1}+\fg_{\pm2e_2}+
\fg_{\pm(e_1+e_2)}+\fg_{\pm(e_1-e_2)}.
\end{eqnarray}
Choose the linearly independent commuting pair $u$ and $v$ from $\fg_{\pm 2e_1}$ and $\fg_{\pm 2e_2}$ respectively.
We claim they satisfy \eqref{zero}.

We can similarly get $\langle u,[u,\ft]\rangle_u^F=0$, and by the reversibility of $F$, $\langle u,\ft\cap\fm\rangle_u^F=0$.
Furthermore, we can prove the root plane decomposition (\ref{root-plane-decomp-1}) is $g_u^F$-orthogonal as well. To see this, consider $g=\diag(z^p,z^q)\in H$, then
  $\Ad(g)$ acts with speed $2p,2 q, p+q, p-q$ on the respective root spaces.
  The conditions for $p$ and $q$ implies $\mathrm{Ad}(H)$ rotates each root plane with a different speed, and as $\Id$ on $\ft\cap\fm$.
  We now choose $g=\diag(z^p,z^q)\in H$ such that $\mathrm{Ad}(g)u=-u$, i.e. $z^{2p}=-1$. Since the Finsler metric is reversible, $\Ad(g)$ preserves $g_u^F$. Thus the root plane decomposition (\ref{root-plane-decomp-1})
  is $g_u^F$-orthogonal, which implies that \eqref{zero} holds.

\bigskip

\subsection{Example 4}
$G=\Sp(3)$ and $H=\Sp(1)_{(3)}\cdot\S^1_{(1,3,0)}=\diag(z,z^3,r)$ with $z\in\C$ and $r\in\Sp(1)$. The argument is very similar to that for Example 1.

Here $\fm$ contains, besides $\ft\cap\fm$, the root planes for $\pm 2e_1,\ \pm 2e_2,\ \pm(e_1\pm e_2),
\ \pm(e_1\pm e_3),\ \pm(e_2\pm e_3) $. Let $g=\diag(z,z^3,z^4)\in\S^1_{(1,3,4)}\subset H$. Then $\Ad(g)$ acts with speed $2,6,4,2,5,3,7,1$   on the root spaces corresponding to the roots $\pm 2e_1$, $\pm2e_2$, $\pm (e_1+e_2)$, $\pm( e_1-e_2)$,
$\pm( e_1+e_3)$, $\pm (e_1-e_3)$, $\pm (e_2+e_3)$, $\pm( e_2-e_3)$ respectively. Choose 2 commuting vectors $u\in \fg_{2e_2}$ and $v\in \fg_{e_1-e_3 }$ and $z$ such that $z^6=-1$. Then, since $F$ is reversible, $\Ad(g)$ preserves $g_u^F$. Furthermore, $\ft\cap\fm$ is orthogonal to all root spaces, and  since $z$ is a $12$th root of unity,   all root spaces are  orthogonal to each other, except that  $\ml \fg_{\pm 2e_1},\fg_{\pm (e_1-e_2)}\mr_u^F$ and $\ml \fg_{\pm( e_1+e_3)},\fg_{\pm( e_2+e_3)}\mr_u^F$ could be nonzero. Similarly as before, we can get $\langle[u,\ft\cap\fm]_\fm,u\rangle_u^F=0$. This is already sufficient for us
to verify \eqref{zero}.

\bigskip

\subsection{Example 5}
 $G=G_2$ and $H=\SU(2)$ corresponding to a short root.

Let $\pm\gamma_1$ be a pair of short roots of $\mathfrak{g}$ and  $\pm\gamma_2$  a pair of long roots  such that the angle between $\gamma_1$ and $\gamma_2$ is
$3\pi/4$. The other roots of $\mathfrak{g}$ are short roots $\pm\gamma_3=\pm(\gamma_1+\gamma_2)$ and $\pm\gamma_4=\pm(2\gamma_1+\gamma_2)$, and long roots
$\pm\gamma_5=\pm(3\gamma_1+\gamma_2)$ and $\pm\gamma_6=\pm(3\gamma_1+2\gamma_2)$.

Denote  by
$$\mathfrak{m}_0=\mathfrak{t}\cap\mathfrak{m}+\mathfrak{g}_{\pm\gamma_6},\quad
\mathfrak{m}_1=\mathfrak{g}_{\pm\gamma_2}+\mathfrak{g}_{\pm\gamma_5},\quad \mathfrak{m}_2=\mathfrak{g}_{\pm\gamma_3}+\mathfrak{g}_{\pm\gamma_4}$$

Then we can find a group element $g\in T\cap H$ and suitable bi-invariant orthonormal bases for each root plane such that
$$
\mathrm{Ad}(g)|_{\mathfrak{m}_0}=\mathrm{Id},\quad
\mathrm{Ad}(g)|_{\mathfrak{m}_1}=-\mathrm{Id},\quad
\mathrm{Ad}(g)|_{\mathfrak{m}_2}=R(\pi/3)
$$
where $R(\theta)$ is the  anti-clockwise rotation with
angle $\theta$. As before, it follows that for all $u\in\fm_1$ the decomposition $\mathfrak{m}=\mathfrak{m}_0+ \mathfrak{m}_1+\mathfrak{m}_2$ is  $g_u^F$-orthogonal.

Take an $F$-unit vector $u\in\mathfrak{m}_1$ such that its bi-invariant length
reaches the maximum among all $F$-unit vectors in $\mathfrak{m}_1$. Apply a suitable $\mathrm{Ad}(\exp\mathfrak{m}_0)$-action, we can make $u$ a vector
in $\mathfrak{g}_{\pm\gamma_2}$. Let $v$ be any nonzero vector in $\mathfrak{g}_{\pm\gamma_4}$.

Direct calculation shows
\begin{equation*}
[u,\mathfrak{m}]_\mathfrak{m}\subset\mathfrak{m}_0+
[u,\mathfrak{t}]+\mathfrak{g}_{\pm\gamma_5}\subset\mathfrak{m}_0+\mathfrak{m}_1,\quad
[v,\mathfrak{m}]_\mathfrak{m}\subset\mathfrak{m}_0+
\mathfrak{m}_2.
\end{equation*}
Our assumption on the choice of $u$ implies
$\langle u,\mathfrak{g}_{\pm\gamma_5}+[u,\mathfrak{t}]\rangle_u^F=0.$
 Thus the conditions in \eqref{zero} are satisfied and we obtain a contradiction.

\bigskip

\section{Reversible $(\alpha,\beta)$ Finsler metrics }

We now exclude some of the exceptions in the case when the reversible homogeneous metric $F$ is  a non-Riemannian $(\alpha,\beta)$ Finsler metric on $G/H$. Recall that according to \cite{DX},
 such a Finsler metric is determined by an $\mathrm{Ad}(H)$-invariant inner product $\ml \cdot,\cdot\ \mr$ with norm $|\cdot|$, a nonzero  vector $v_0\in\fm$ fixed by $\mathrm{Ad}(H)$,  and a positive even function $\phi$ such that  $F(v)=|v| \phi(\ml v_0,v/|v|\mr)$.

Let $\fm=V_1\oplus V_2$ where $V_2$ is the span of $v_0$, and $V_1$ its orthogonal complement with respect to $\ml \cdot,\cdot\mr$. This decomposition is preserved by the action of $\mathrm{Ad}(H)$. Further more, for any $k\in \O(V_1)$, we have $F(kv)=F(v)$. Thus the restriction of $F$ to $V_1$
is Riemannian, and for all $u\in V_1$ the inner product
 $\langle \cdot,\cdot\rangle_u^F=(D^2F)_u(\cdot,\cdot)$ is independent of $u$. We denote this inner product by $\langle \cdot,\cdot\rangle_0$. It is $\mathrm{Ad}(H)$-invariant, so it defines a $G$-invariant Riemannian metric
 on $G/H$. Comparing the homogeneous flag curvature formula in \pref{flagcurv} for $F$ with the one for $\langle\cdot,\cdot\rangle_0$,  we obtain the following immediate consequence:

\begin{prop} \label{flagabcurv}
Let $(G/H,F)$ be a non-Riemannian homogeneous $(\alpha,\beta)$ reversible Finsler manifold. Then for any linearly independent vectors
 $u,v\in V_1$ with $[u,v]=0$ and $\langle[u,\fm],u\rangle_u^F=0$, the Flag curvature for $u$ and $u\wedge v$ is given by:
\begin{equation*}
K^F(o,u,u\wedge v)=K^0(o,u\wedge v)
\end{equation*}
where $K^0$ is the sectional curvature of the Riemannian metric $\langle \cdot,\cdot\rangle_0$.
\end{prop}

By Theorem A and its Corollary, we can exclude all homogeneous spaces, except for the finite list of examples in Theorem A.
So the discussion for the following 3 examples ends the proof of Theorem B.

 \subsection{Example 1}

$\SU(4)/\diag(zA,z,\bar z^3)$ with $A\in\SU(2)$ and $z\in\C$. We have the following  $\Ad_H$ invariant decomposition: $$\mathfrak{m}_0=\mathfrak{t}\cap\mathfrak{m}+
\mathfrak{g}_{\pm(e_3-e_4)},\ \mathfrak{m}_1=\mathfrak{g}_{\pm(e_1-e_3)}+\mathfrak{g}_{\pm(e_2-e_3)},\ \mathfrak{m}_2=\mathfrak{g}_{\pm(e_1-e_4)}+\mathfrak{g}_{\pm(e_2-e_4)}$$
Notice that $\Ad_H$ acts by inequivalent irreducible representations on each of $\fm_1$ and $\fm_2$ and both are inequivalent to the representation on $\fm_0$. So $\fm_1+\fm_2\subset V_1$, restricted to which $F$ coincides with $\langle\cdot,\cdot\rangle_0$. Further more, $\langle\fm_1,\fm_2\rangle_0=0$, and the restrictions of $F$ to $\fm_1$ and $\fm_2$ are multiples of the biinvariant metric. So we have $\langle[u,\fm],u\rangle_u^F=0$ for any nonzero vector $u\in \mathfrak{g}_{\pm(e_1-e_3)}$. Choose the nonzero vector $v\in \mathfrak{g}_{\pm(e_2-e_4)}$. Then for the homogeneous Riemannian metric $\langle \cdot,\cdot\rangle_0$, we see $u$ and $v$ are 2 commuting eigenvectors and hence this plane has 0 curvature for $\langle \cdot,\cdot\rangle_0$, and by Proposition \ref{flagabcurv}, also for $F$.

\bigskip

\subsection{Example 2}

$G_2/\SU(2)$ with $\SU(2)$ the normal subgroup of $\SO(4)$ corresponding to the long root. For the $\Ad(H)$-actions, $\fm$ splits into the sum of a trivial 3 dimensional representation $\fm_0$ and two equivalent irreducible 4 dimensional representation $\fm_1\oplus\fm_2$. Notice $\fm_1+\fm_2\subset V_1$, restricted to which $F$ coincides with $\langle\cdot,\cdot\rangle_0$. For any nonzero eigenvector $u\in\fm_1+\fm_2$ for $\langle\cdot,\cdot\rangle_0$, we have $\langle [u,\fm]_\fm,u\rangle_u^F=0$.
Notice $\langle [u,\fm_i]_\fm,u\rangle_u^F=0$ is obvious for $i=1,2$, but it
requires $u$ to be an eigenvector for $i=0$. Thus by \pref{flagabcurv} we have a contradiction if we can find a 0 curvature plane for $\langle\cdot,\cdot\rangle_0$ containing $u$. The argument from Section 5.9 of \cite{WZ} or from \cite{BB}  provides the required vector $v$.

\subsection{Example 3}  $\Sp(2)/\diag(z,z^3)$. In 5.14 of \cite{WZ}, it was shown that we can find $$g=\mathrm{diag}(e^{\theta\mathbf{i}}\mathbf{j}\,
e^{-\theta\mathbf{i}},e^{\theta\mathbf{i}}\mathbf{j}\,
e^{-\theta\mathbf{i}})\in\mathrm{Sp}(2)$$ for a suitable $\theta\in\mathbb{R}$ such that the $\mathrm{Ad}(g)$-action on $\mathfrak{m}$
preserves the the $\mathrm{Ad}(H)$-invariant inner product
$\langle\cdot,\cdot\rangle$ on $\mathfrak{m}$. Obviously
${V}_2=\mathfrak{t}\cap\mathfrak{m}$, on which $\mathrm{Ad}(g)$ acts as $-\mathrm{Id}$. So $\mathrm{Ad}(g)$ preserves the reversible $(\alpha,\beta)$ Minkowski norm $F$ on $\mathfrak{m}$. The homogeneous Finsler metric $F$ naturally induces a homogeneous Finsler metric $\tilde{F}$ on $G/\tilde{H}$ where $\tilde{H}=H\cup gH$ and is also positively curved. On the other hand, $G/\tilde{H}$ is not orientable since
$\mathrm{Ad}(g)$ reverses the orientation on $\mathfrak{m}$. This is a contradiction to Synge Theorem in Finsler geometry for the odd dimensional case.

\bigskip

\providecommand{\bysame}{\leavevmode\hbox
to3em{\hrulefill}\thinspace}

\end{document}